\documentclass[12pt]{article}

\usepackage{amstext}
\usepackage{amsthm}
\usepackage{amsmath}
\usepackage{amssymb}
\usepackage{latexsym}
\usepackage{amsfonts}
\usepackage{color}
\usepackage{graphicx}
\usepackage{yfonts}
\usepackage{verbatim}
\usepackage{enumitem}

\AtBeginDocument{%
 \def\MR#1{} 
}

\usepackage[pagebackref,hypertexnames=false, colorlinks, citecolor=black, linkcolor=black, urlcolor=red]{hyperref}
\usepackage[backrefs]{amsrefs}
\usepackage{latexsym}
\usepackage{amssymb}
\usepackage{euscript}

\let\cal=\mathcal      

\def\mcc{M\raise.5ex\hbox{c}C}
\def\mccarthy{M\raise.5ex\hbox{c}Carthy}
\def\Hu{\H}
\def\sz{Szeg\Hu{o} }

\def\ie{{\it i.e. }}

\def\h{{\cal H}}
\def\hk{{\cal H}_k}
\def\hs{{\cal H}_s}

\def\M{{\cal M}}

\def\m{Mult}

\def\l{\lambda}

\def\vare{\varepsilon}

\let\i=\infty

\def\la{\langle}
\def\ra{\rangle}
\def\={\ = \ }

\def\F{{\cal F}}

\def\E{E_\l}

\def\C{\mathbb C}

\def\D{\mathbb D}
\def\B{\mathbb B}

\def\dis{\displaystyle}

\def\be{\setcounter{equation}{\value{theorem}} \begin{equation}}
\def\ee{\end{equation} \addtocounter{theorem}{1}}
\def\beq{\begin{eqnarray*}}
\def\eeq{\end{eqnarray*}}

\def\bp{{\sc Proof: }}
\def\ep{{}{\hfill $\Box$} \vskip 5pt \par}

\def\bl{\begin{lemma}}
\def\el{\end{lemma}}
\def\bt{\begin{theorem}}
\def\et{\end{theorem}}
\def\bprop{\begin{prop}}
\def\eprop{\end{prop}}
\def\bd{\begin{definition}}
\def\ed{\end{definition}}
\def\br{\begin{remark}}
\def\er{\end{remark}}
\def\bexer{\begin{exercise}}
\def\eexer{\end{exercise}}

\newtheorem{theorem}{Theorem}[section]
\newtheorem{proposition}[theorem]{Proposition}
\newtheorem{lemma}[theorem]{Lemma}

\newtheorem{definition}[theorem]{Definition}

\newtheorem{question}[theorem]{Question}

\newtheorem{example}[theorem]{Example}

\newcommand\hkl{{\hat{k}_\lambda}}

\newcommand\mn{{\mathbb M}_n}
\newcommand\cB{{\mathcal B}}
\renewcommand\L{\mathcal L}
\newcommand\hl{\hat{\ell}}

\newcommand\cp{s}
\newcommand{\cH}{\mathcal H}
\newcommand{\bN}{\mathbb N}
\DeclareMathOperator{\Mult}{Mult}
\renewcommand\E{\mathcal E}
\renewcommand{\m}{\operatorname{Mult}}

\numberwithin{equation}{section}
\title{Interpolating sequences in spaces with the complete Pick property}
\author{Alexandru Aleman
\\ Lund University
\and
Michael Hartz
\thanks{Partially supported by an Ontario Trillium Scholarship and a Feodor Lynen Fellowship}
\\ Washington University 
\and
John E. M\raise.5ex\hbox{c}Carthy
\thanks{Partially supported by National Science Foundation Grant
DMS 1565243}
\\ Washington University 
\and
Stefan Richter
\\ University of Tennessee
}

\begin{document}

\bibliographystyle{plain}
\maketitle

\begin{abstract}
  We characterize interpolating sequences for multiplier algebras of spaces with the complete Pick property.
  Specifically, we show that a sequence is interpolating if and only if it is  separated
  and generates a Carleson measure. This generalizes results of Carleson for the Hardy space
  and of Bishop, Marshall and Sundberg for the Dirichlet space. Furthermore,
  we investigate interpolating sequences for pairs of Hilbert function spaces.
\end{abstract}

\section{Introduction}
\label{seca}

Let $\h$ be a reproducing kernel Hilbert space on a set $X$. In this paper we will always assume that the reproducing kernel $k$ satisfies $\|k_z\|^2=k_z(z)\ne 0$ for all $z\in X$.
Let $\m(\h)$ denote the multiplier algebra  of $\h$, that is the set of functions
$\phi$ on $X$ with the property that multiplication by $\phi$ maps $\h$ to $\h$.
A sequence $(\lambda_i) \subseteq X$ is called an {\em interpolating sequence} for $\m(\h)$  if,
whenever $(w_i)$ is a bounded sequence of complex numbers, there is a multiplier $\phi$
such that $\phi(\lambda_i) = w_i$ for each $i$. Furthermore, the sequence  $(\lambda_i) \subseteq X$ is called an {\em interpolating sequence} for $\h$  if the operator $$
T: f \mapsto\left(\frac{f(\lambda_i)}{\|k_{\lambda_i}\|}\right) $$
 takes $\h$ boundedly onto $\ell^2$.

L. Carleson classified interpolating sequences for $H^\infty,$ the multiplier algebra of the Hardy space $H^2$ on the unit disk, in 1958
\cite{car58}. H.S. Shapiro and A. Shields found a different proof of Carleson's theorem and showed that the interpolating sequences for $H^\infty$ are the same as the ones for $H^2$ \cite{SS61}. This idea was used
 by D. Marshall and C. Sundberg \cite{marsun} and C. Bishop \cite{bis} to characterize interpolating sequences
in the context of the Dirichlet space on the unit disk.

The key property shared by the Hardy space and the Dirichlet space that allowed this translation to work
is the Pick property. We define an irreducible complete Nevanlinna-Pick kernel in Section~\ref{secb},
but by Theorem \ref{thmb1}, it  is a reproducing kernel $s$ of the form
\begin{equation}
\label{ncp}
s_w(z)=\frac{f(z)\overline{f(w)}}{1-\langle b(z),b(w)\rangle_{\mathcal K}}
\end{equation}
where $f:X\to \C$ is  non-vanishing and $b$ is a function from $X$ into the open unit ball of an auxiliary Hilbert space ${\mathcal K}$.
If, in addition, $f$ is identically $1$ and there is a point $w_0 \in X$ such that
$b(w_0) = 0$, we will call \eqref{ncp}
a {\em normalized complete Pick kernel}.
The normalization is convenient for our proofs, but is not essential.
The Hardy and Dirichlet spaces of the unit disc are examples of spaces with normalized complete Pick kernels. In Section \ref{secb} we will provide  further examples as well as motivation and background about reproducing kernels of this type.

Interpolating sequences are characterized by  separation  and  Carleson measure conditions.
If $\h$ is a reproducing kernel Hilbert space with kernel $k$, then $$d_\h(z,w) = \sqrt{1-\frac{|\la k_z,k_w\ra|^2}{\|k_z\|^2\|k_w\|^2}},\ \ z,w\in X$$ defines a pseudo metric on $X$. Indeed, it is clearly symmetric, it satisfies the triangle inequality (see \cite[Lemma 9.9]{ampi}), and it satisfies $d_\h(z,w)=0$ if and only if $k_z$ and $k_w$ are linearly dependent. As an  example consider $\h=H^2$, then $d_{H^2}(z,w)=|\frac{z-w}{1-\overline{z}w}|$ is the pseudohyperbolic metric on the unit disk. The sequence $(\lambda_i)$ is called {\em $\h$-separated }, if there is $c>0$ such that
\[\tag{S}d_\h(\lambda_n,\lambda_m) \ge c>0\ \ \text{for all } \ n\ne m.\]

We shall say the sequence is strongly separated (SS) if
 there exists $\vare > 0$ such that
for every $ i$ there exists a function $\phi$ in the unit ball of $\m(\h)$ such that
$\phi(\l_i) = \varepsilon$ and $\phi(\l_j) = 0 \ \forall\ j \neq i$. It is easy to see that for $\h=H^2$ condition (SS) is equivalent to the existence of $\vare>0$  such that for each $i$ $\prod_{i\ne j}|\frac{\lambda_i-\lambda_j}{1-\overline{\lambda_j}\lambda_i}|\ge \vare$.

Finally, the sequence $(\lambda_i)$ satisfies  the Carleson measure condition if
there exists $c > 0$ such that
\[\tag{C}\sum_i\frac{|f(\lambda_i)|^2}{\|k_{\lambda_i}\|^2} \le c\|f\|^2, \ \ f\in \h.\]
If $\delta_\lambda$ denotes the unit point mass at the point $\lambda\in X$, then this condition is saying that $\mu= \sum_i \frac{1}{\|k_{\lambda_i}\|^2} \delta_{\lambda_i}$ is a Carleson measure for the Hilbert space $\h$.

Thus we have the following conditions that a sequence $(\lambda_i)\subseteq X $ may satisfy for a given space $\h$:
\begin{itemize}

 \item[(IM)] the sequence is interpolating for  $\m(\h)$,

\item[(IH)] the sequence is interpolating for  $\h$,

 \item[(S+C)] \quad the sequence is $\h$-separated and satisfies the Carleson measure condition for $\h$,

 \item[(SS)] the sequence is strongly separated.
 \end{itemize}

Carleson proved that for $\h=H^2$ conditions (IM), (S+C), and (SS) are equivalent, and, of course, he also established a geometric Carleson measure condition for $H^2$. Shapiro and Shields added to this the equivalence of (IH).

By now it is well-known  that the implications (IM) $\Rightarrow$ (IH) and (IH) $\Rightarrow$ (S+C) hold for every reproducing kernel Hilbert space (see \cite[Chapter 9]{ampi}). The implication (IH) $\Rightarrow$ (IM)  for all irreducible complete Pick kernels was proved by Marshall-Sundberg \cite{marsun} and by Bishop \cite{bis}. Both Marshall-Sundberg and Bishop also construct examples to show that (SS) does not imply (IM) in the case of the Dirichlet space of the unit disc. In \cite[Theorem 9.43]{ampi} it is shown  that (S+C) implies (SS) for all irreducible complete Pick kernels.

The main result of this note addresses the remaining open implication in the generality of spaces with the complete Pick property.

\begin{theorem}\label{CNPintThm} The implication (S+C) $\Rightarrow$ (IM) holds for all reproducing kernel Hilbert spaces with irreducible complete Nevanlinna-Pick kernels.
\end{theorem}
We will also see that under this hypothesis the interpolating functions can  be found via a bounded linear interpolation operator $T:\ell^\infty\to \m(\h)$, i.e. there
 is a sequence of multipliers $\phi_i\in \m(\h)$ such that for each sequence $(w_i)\in \ell^\infty$ the function $\phi=\sum_i w_i \phi_i$ satisfies $\phi \in \m(\h)$ and  $\phi(\lambda_i)=w_i$ for each $i$.

 In \cite{boe05} B. B\"oe showed that (S+C) $\Rightarrow$ (IM) holds for some spaces of analytic functions, if the Gram matrix $\left( \langle \frac{k_{\lambda_n}}{\|k_{\lambda_n}\|}, \frac{k_{\lambda_m}}{\|k_{\lambda_m}\|} \rangle \right)_{1\le {n,m}<\infty}$ of the
 complete Nevanlinna-Pick kernel $k$ satisfies an extra hypothesis. While B\"oe's theorem applied to spaces of analytic functions in the unit ball of $\C^d$ with reproducing kernel $(1-\langle z, w\rangle)^{-\gamma}$, $0<\gamma<1$, up to now it had been an open question whether the Theorem holds for the Drury-Arveson kernel $(1-\langle z, w\rangle)^{-1}$. In fact, Theorem \ref{CNPintThm} answers a question of Agler-M\raise.4ex\hbox{c}Carthy  \cite[Question 9.57]{ampi} and it affirms a conjecture of Seip  \cite[Conjecture 1, p. 33]{SeipIntSamp} at least for spaces with the complete Pick property.

We will now define the concept of interpolating sequences for multipliers between spaces.

Let $k,\ell$ be two reproducing kernels on a set $X$ such that $k_z, \ell_z \ne 0$ for all $z\in X$. We will denote the corresponding reproducing kernel Hilbert spaces by $\h_k$ and $\h_\ell$. If $\varphi \in \m(\h_k, \h_\ell)$, then it is easily seen that the function $\varphi$ satisfies a growth estimate:
$$ |\varphi(z)|=\frac{|\varphi(z) k_z(z)|}{\|k_z\|^2}=\frac{|\langle \varphi k_z, \ell_z\rangle|}{ \|k_z\|^2}\le \|\varphi\|_M\frac{\|\ell_z\|}{\|k_z\|},
$$ where we have written $\|\varphi\|_M$ for the multiplier norm of $\varphi$, i.e. the norm of the multiplication operator $M_\varphi: \h_k\to \h_\ell$.
 We say that a sequence $(\lambda_i) \subseteq X$ is {\em interpolating for $\m(\h_k, \h_\ell)$}, if whenever $(\alpha_i)$ is  a bounded sequence of complex numbers, then there is a $\varphi \in \m(\h_k, \h_\ell)$ with $\varphi(\lambda_i)=\alpha_i\frac{\|\ell_{\lambda_i}\|}{\|k_{\lambda_i}\|}$.

In Section \ref{sece} we will investigate interpolating sequences for $\m(\h_k, \h_\ell)$.

\begin{theorem} \label{MixedInterpolation} Let $s$ be a normalized  complete Pick kernel on $X$ and $\ell=gs$, where $g$ is positive semi-definite, i.e. for all $\lambda_1,..., \lambda_n \in X$ and all $a_1,..., a_n \in \C$ we have $\sum_{i,j} a_i\overline{a}_j g_{\lambda_i}(\lambda_j) \ge  0$.

(a) A sequence is interpolating for $\m(\h_s, \h_\ell)$ if and only if the sequence satisfies the Carleson condition (C) for $\h_s$ and it is interpolating for $\h_\ell$.

(b) If a sequence is $\h_s$-separated, then it is interpolating for $\m(\h_s, \h_\ell)$ if and only if is interpolating for $\m(\h_s)$.
\end{theorem}
Note that by taking $\ell=s$ we recover the equivalence of (IM) and (IH) for irreducible complete Pick kernels. This theorem applies for example, whenever $\h_s=H^2$ and the operator $M_z$ of multiplication with the coordinate function  defines a contraction operator on $\h_\ell$.

In certain situations we can improve the previous theorem. We start with two simple observations. First by Remark 8.10 of \cite{ampi} and the Schur product theorem the expression $s^t_w(z)$ defines a reproducing kernel, whenever $s$ is a normalized  complete Pick kernel and $t>0$. Second, since for fixed $t>0$ we have $\sqrt{1-x^t}\sim \sqrt{1-x}, 0\le x\le 1$, it follows that a sequence is $\h_s$ separated, if and only if it is $\h_{s^t}$ separated.

\begin{theorem} \label{thmsst} Let $s_1,s_2$ be  normalized  complete Pick kernels on $X$ such that $s_2/s_1$ is positive semi-definite, and  let $t\ge 1$.

Then a sequence is interpolating for $\m(\h_{s_1},\h_{s_2^t})$ if and only if it satisfies the $\h_{s_1}$ Carleson condition (C) and is $\h_{s_2}$-separated (S).
\end{theorem}

In particular, it follows from Theorem \ref{CNPintThm} that for a normalized complete Pick kernel $s$ and any $t\ge 1$, the $\m(\h_s,\h_{s^t})$ interpolating sequences are just the $\m(\h_s)$ interpolating sequences.

Theorem \ref{thmsst} applies if $s_1$ is any normalized  irreducible complete Nevanlinna Pick kernel for a space of analytic functions in the open unit disk $\D$ and if ${s_2}_w(z)=\frac{1}{1-\overline{w}z}$ is the Szeg\"{o} kernel, the reproducing kernel for $H^2$. In this case $s_1(z)=\frac{1}{1-\sum_{n\ge 1} \overline{b_n(w)}b_n(z)}$ for some analytic functions $b_n$ with $\sum_{n \ge 1}|b_n(z)|<1$ for all $|z|<1$. This implies that the row operator $(f_1, f_2,...)\to \sum_{n \ge 1} b_nf_n$ is bounded from $H^2\oplus H^2 \oplus ...$ into $H^2$, and this is easily seen to be equivalent to $s_2/s_1$ being positive semi-definite.

Then for any $t>1$ the space $\h_{s_2^t}$ is a weighted Bergman space of analytic functions on the open unit disc such that $\int_{\mathbb{D}}|f|^2(1-|z|^2)^{t-1} dA<\infty$.  In Example \ref{ExNotSeparated} we will show that if $\h_\ell$ is a weighted Bergman space with weight $\exp(-\frac1{1-|z|^2})$, then there are $\m(H^2, \h_\ell)$ interpolating sequences that are not $H^\i$-interpolating.

Section \ref{secf} contains applications of our results. In particular, using Theorem \ref{CNPintThm},
we generalize a result of Lech \cite{LechEssNormal}; he showed that
 there is  a multiplication operator
on the Dirichlet space which is not essentially normal, and in Proposition \ref{prxe1} we show that this result is true under a mild additional hypothesis
in any complete Pick space where the kernel function is unbounded. Moreover, we indicate how the results of Section \ref{sece}
can be used to exhibit noncompact multiplication operators between pairs of spaces.

\section{Background}

\subsection{Complete Pick kernels}
\label{secb}

Let $\h_k$ be a reproducing kernel Hilbert space on a set $X$, with kernel $k$. When referring to reproducing kernels we will use the notations $k_w(z)=k(z,w)$ interchangeably.  For basic facts about reproducing kernel Hilbert spaces we refer the reader to \cite{ampi} or \cite{para16}. Let $n$ be a positive integer, and let $\mn$ denote the $n$-by-$n$ complex matrices.
We say that $k$ has the {\em $\mn$ Pick property} if, for every finite sequence
$\l_1, \dots , \l_N$ of distinct points in $X$, and every sequence $W_1, \dots, W_N$ in $\mn$,
the 
operator
\beq
L : \vee\{ k_{w_j}  \otimes \C^n \ | \ 1 \leq j \leq N \}   &\to & \vee\{ k_{w_j}  \otimes \C^n \ | \ 1 \leq j \leq N \} \\
k_{w_j} \otimes v & \mapsto & k_{w_j} \otimes W_j^* v
\eeq
has a norm-preserving  extension to the adjoint of a multiplication operator
\[
M_{\Phi}^* : \h_k \otimes \C^n \to \h_k \otimes \C^n .
\]
This is equivalent to saying that whenever
the block matrix
\begin{equation*}
  \Big[ k(\lambda_j,\lambda_j) (I_{\mathbb C^n} - W_i W_j^*) \Big]_{i,j=1}^N
\end{equation*}
is positive, there is a multiplier $\Phi$ of $\h_k \otimes \C^n$ of norm at most $1$ that satisfies
\[
\Phi(\l_i) \= W_i, \quad \forall \ 1 \leq i \leq N .
\]
When $n=1$, we say $k$ has the {\em  Pick property}.
If $k$ has the $\mn$ Pick property for every $n$, we say the kernel, and the corresponding Hilbert space
$\h_k$,  have the {\em complete Pick property}
or the {\em complete Nevanlinna-Pick property}.

Kernels with the  complete Pick property are described by the McCullough-Quiggin theorem
\cite{mccul92, qui93}; an alternative description is in \cite{agmc_loc}.
Let us say that the kernel $k$ is
{\em irreducible}  if $X$ cannot be partitioned into two disjoint non-empty sets $X_1$ and $X_2$ such that
$k(x_1,x_2) = 0$ whenever $x_1 \in X_1, x_2 \in X_2$.
The McCullough-Quiggin theorem says that an irreducible kernel has the complete Pick property if and only if
$1/k$ has exactly one positive square, \ie for every finite set of distinct points $\{ \l_1, \dots, \l_N \} \subset X$,
the self-adjoint  matrix
\[
\left[ \frac{1}{k(\l_j, \l_i)} \right]_{i,j =1}^N
\]
has exactly one positive eigenvalue, counted with multiplicity.

Examples of such kernels and spaces are the \sz kernel $\dis \frac{1}{1- \bar w z}$  for the Hardy space on the unit disk; the Dirichlet kernel $\dis \frac{-1}{\bar w z} \log(1 - \bar w z)$  on the disk \cite{shashi62, marsun};
the kernels $\dis \frac{1}{(1- \bar w z)^t}$ for $0 < t < 1$ on the disk  \cite{shashi62, marsun};
the Sobolev space $W^2_1$ on the unit interval \cite{ag90b};  various weighted Sobolev spaces \cite{qui93}; superharmonically weighted Dirichlet spaces in the unit disc \cite{Shimorin};
 the kernel $\dis \frac{1}{2 - \zeta( z + \bar w)}$
on a half-plane
\cite{mcc04}; and the Drury-Arveson space, the space of analytic functions on the unit ball $\B_d$ of a $d$-dimensional
Hilbert space (where $d \in \{ 1,2, \dots, \i \})$ with kernel
\[
k(z,w) \= \frac{1}{1 - \la z, w \ra}
\]
\cite{lu76, dru78, arp95, dapi98, agmc_loc, arv98}.

If $\l_0$ is a point in $X$, and $k(z, \l_0) \neq 0$ for all $z \in X$, one can define a new kernel
\[
\ell(z,w) \= \frac{k(z,w) k(\lambda_0,\lambda_0)}{k(z,\l_0) k(\l_0, w) }.
\]
The new kernel will have the property that
$\ell(z, \l_0) = 1$ for all $z \in X$ (we call this being normalized at $\l_0$) and
$\h_\ell$ is just $k(\lambda_0,\lambda_0)^{1/2} k(\cdot, \l_0)^{-1} \h_k$.
For an irreducible Pick kernel, $k(z, w)$ is never 0 \cite[Lemma 7.2]{ampi}, so it is convenient to assume
that our kernels are always normalized.

The Drury-Arveson space is universal among all spaces that have the complete Pick property, in the following sense
\cite{agmc_cnp}.
\bt
\label{thmb1}
Let $\cp$ be an irreducible kernel on $X$ with the complete Pick property, normalized at $\l_0$.
Then there is an auxiliary Hilbert space ${\mathcal K}$
and  there is a map $b : X \to {\mathcal K}$ with $b(\l_0)=0$ and $\|b(z)\|_{\mathcal K}<1$ for all $z\in X$
such that
\[
\cp(z,w) \= \frac{1}{1 - \la b(z), b(w) \ra_{\mathcal K}} .
\]
If $\h$ is separable, then ${\mathcal K}$ can be chosen to be separable.
\et

\subsection{Grammians}
\label{secc}

Let $\h_k$ be a reproducing kernel Hilbert space on a set $X$, with kernel $k$. We will write $\hat{k}_z=k_z/\|k_z\|$ for the normalized kernel function at $z$.
Let $(\l_i)$ be a sequence of distinct points in $X$. The Grammian, or Gram matrix, associated with the
sequence is the (infinite) matrix $G(k) = [G_{i,j}]$, where
\[
  G_{ij} \ :=  \ \langle \hat k_{\lambda_i}, \hat k_{\lambda_j} \rangle  \= \frac{k(\l_j,\l_i)}{\sqrt{k(\l_j, \l_j) k(\l_i, \l_i)}} .
\]
If the kernel $k$ is understood, simply write $G$ instead of $G(k)$.

We shall say that the sequence  $(\l_i)\subseteq X$ has a bounded Grammian (BG) if the Gram matrix, thought of as an operator
on $\ell^2$, is bounded; we shall say that it is bounded below (BB) if the Gram matrix is bounded below on
$\ell^2$. It is known that  if the Grammian  of a sequence for the Szeg\"{o} kernel $s_w(z)=\frac{1}{1-\overline{\lambda}z}$ is bounded below, then it is bounded above. The analogous fact is not true for general complete Nevanlinna-Pick kernels, see \cite[Example 9.55]{ampi}.

In the following lemma we have listed some elementary facts about Grammians.

\begin{lemma} \label{elementaryGrammian} (a) The Grammian is bounded (BG) if and only if the sequence satisfies the Carleson measure condition for $\h_k$ (condition (C) from the Introduction).

(b) The  following three conditions are equivalent:
\begin{enumerate}
\item the Grammian
is bounded and bounded below (BG+BB),
\item the functions $\hat{k}_{\lambda_i}$ form a Riesz sequence, i.e. there are $c_1,c_2>0$ such that for all scalars $a_i$
$$c_1 \sum_i|a_i|^2 \le \|\sum_i a_i \hat{k}_{\lambda_i}\|^2 \le c_2 \sum_i|a_i|^2,$$
\item the sequence is interpolating for $\h_k$ (condition (IH) from the Introduction).
\end{enumerate}
\end{lemma}

For the proof of (a) see \cite[Proposition 9.5]{ampi}. The equivalence of (1) and (2) in (b) is \cite[Proposition 9.13]{ampi}, while the equivalence of (2) and (3) follows by considering the adjoint of the interpolation map $Tf= (\frac{f(\lambda_i)}{\|k_{\lambda_i}\|})$.

To prove Theorem~\ref{CNPintThm} we shall use the following result from \cite[Theorem 9.46]{ampi};
we use $\{e_i \}$ to denote the standard
orthonormal basis for $\ell^2$.

\bt
\label{thmd2}
Let $k$ be an irreducible complete Pick kernel on $X$, let
$(\l_i )$ be a sequence in $X$, and let $G$ denote the Grammian associated with $(\lambda_i)$.
Then

{\rm
(a)} $G$ is bounded below if and only if there is a function $\Psi$
in $\m(\hk \otimes \ell^2, \hk )$ such that
\be
\label{eqie6}
\Psi(\l_i) \= e_i^\ast \= (0, \dots, 0 ,
1 ,
0 , \dots ) .
\ee

{\rm
(b)} $G$ is bounded and the sequence is weakly separated
if and only if there is a function $\Psi$
in $\m(\hk, \hk \otimes \ell^2 )$ such that
\be
\label{eqie8}
\Psi(\l_i) \= e_i \=
\left( \begin{array}{c}
0\\
\vdots\\
0\\
1\\
0\\
\vdots
\end{array} \right) .
\ee
\et

\begin{remark}
  In \cite{ampi}, a more restrictive definition of irreducibility is used.
  In addition to our irreduciblity assumption, it is assumed
  that $k(\cdot,w)$ and $k(\cdot,z)$ are linearly independent if $z \neq w$.
  However, the proof of \cite[Theorem 9.46]{ampi} shows that this theorem remains
  valid under our less restrictive definition of irreducibility.
\end{remark}

\section{The proof of Theorem \ref{CNPintThm} and some remarks}
\label{secd}

We are now in  position to prove Theorem \ref{CNPintThm}.

\bp Let $s_w(z)$ be a complete Pick kernel on $X$, and assume that the sequence $(\lambda_i)$ is $\h_s$ separated and satisfies the Carleson condition (C). We must prove that the sequence is interpolating.

By Lemma \ref{elementaryGrammian}(a) the Gram matrix is bounded.

 By the  Marcus-Spielman-Srivastava theorem \cite{mss15},
 since the Gram matrix $G$ is bounded, after a permutation the matrix $G-I$ can be decomposed into
 a finite block matrix such that each diagonal block
  has norm at most $1/2$ (this follows from the paving conjecture, which J. Anderson
  proved \cite{an79} is equivalent to the Kadison-Singer conjecture,
which was proved by the Marcus-Spielman-Srivastava theorem).
By Lemma \ref{elementaryGrammian} (b), this means that the sequence
 $(\l_i)$ is a finite union of  sequences $(z_{i}^{(1)}), \ldots, (z_i^{(n)})$, for each of which
 the Gram matrix is bounded above and below, so each $(z_i^{(j)})$ is an interpolating sequence.

By  Theorem \ref{thmd2} (a),  for $j =1, \dots, n$ we can find a Hilbert space $\L_j$
with an orthonormal basis $\{ v_i^{(j)} \}_{i=1}^\i$ and a bounded multiplier
$\psi_j \in \m(\h \otimes \L_j, \h)$ such that $\psi_j (z_i^{(j)}) = ( v_i^{(j)})^*$.
 Let $\Psi = (\psi_1,\ldots,\psi_n)$, which
  is a bounded multiplier in $\m(\h \otimes ( \oplus_{j=1}^n \L_j ) , \h)$.
  Observe that for each $i \in \mathbb N$, there exists a unique element
  $(\mu(i),\nu(i)) \in \{1,\ldots,n\} \times \mathbb N$ such that $\lambda_i = z_{\nu(i)}^{\mu(i)}$.
  We identify $ \oplus_{j=1}^n \L_j $ with $\ell^2$ in such a way
  that $e_i$ is identified with $v^{\mu(i)}_{\nu(i)}$.
Then $\Psi$ is a bounded multiplier in $\m(\h \otimes \ell^2, \h)$ that satisfies
\be
\label{eqd1}
\Psi(\l_i) \= ( * \cdots * 1 * \cdots ) ,
\ee
where the $1$ is in the $i^{\rm th}$ slot.

 By Theorem \ref{thmd2} (b), applied to the whole sequence $(\l_i)$,
  there exists a bounded multiplier $\Phi \in \m(\h, \h \otimes \ell^2 )$
  such that
\be
\label{eqd2}
\Phi(\l_i) \ =\  e_i \=  \begin{pmatrix}
      0 \\ \vdots \\ 0 \\ 1\\ 0 \\ \vdots
    \end{pmatrix}.
\ee

 Let $\Delta$ be the embedding of $\ell^\infty$ into
  $\cB(\ell^2)$ via diagonal operators and define
 \beq
    T: \ell^\infty &\ \to \ & \m(\h), \\
a &\mapsto & \Psi \Delta(w) \Phi.
\eeq
  It is clear that the range of $T$ consists of multiplication operators
  and that $T$ is linear and bounded. Moreover, if $w = (w_i) $ is any sequence in $ \ell^\infty$,
  then
\begin{align*}
    T(w) (\l_i)
    &=
    \begin{pmatrix}
      * & \cdots & * & 1 & * &\cdots
    \end{pmatrix}
    \begin{pmatrix}
    w_1 & 0 & \ldots  \\
      0 & w_2 & \ldots  \\
      \vdots & \ddots & \ddots  \\
    \end{pmatrix}
    \begin{pmatrix}
      0 \\ \vdots \\ 0 \\ 1\\ 0 \\ \vdots
    \end{pmatrix} \\
    &=
    w_i.
  \end{align*}
Thus $T(w)$ is a multiplier that solves the interpolation problem $\l_i \mapsto w_i$,
so $(\l_i)$ is an interpolating sequence as required.
\ep

\begin{remark}
One application of Theorem~\ref{CNPintThm} is to give another proof of the Marshall-Sundberg and Bishop theorem
\cite{marsun, bis} characterizing the interpolating sequences for the Dirichlet space, and the spaces with kernels
$1/(1-\bar w z)^t$ for $0 < t < 1$, on the unit disk. Two other proofs of this have been given by B. B\"oe \cite{boe02,
boe05}.
\end{remark}

\begin{remark}
The proof actually constructs a linear interpolation operator $T : \ell^\i \to \m(\h)$.
In the case of $H^\i(\D)$, this is P. Beurling's theorem --- see \cite[Theorem VII.2.1]{gar81};
it was proved for the Dirichlet space by B\"oe \cite{boe02}.
\end{remark}

\begin{remark} Interpolating sequences for spaces with irreducible complete Nevanlinna Pick kernels are ``complete interpolating sequences". Indeed, if $W_i\in {\mathcal B}({\mathcal K})$ with $\|W_i\|\le C$ for all $i$, then one checks as in the proof above that $\varphi=\sum_i W_i \psi_i \phi_i \in \m(\h_s\otimes {\mathcal K})$ and it satisfies $\varphi(\lambda_i)=W_i$ for all $i$. Here $\psi=(\psi_1,\psi_2,...), \phi=(\phi_1, \phi_2,...)^t$ are the bounded row and column multiplication operators which satisfy $\psi_i(\lambda_j)=\phi_i(\lambda_j)=\delta_{ij}$ and whose existence is guaranteed by Lemma \ref{elementaryGrammian} and Theorem \ref{thmd2}.
\end{remark}

\begin{remark}
In the papers \cite{os08, sei09, ol11} interpolating sequences are characterized for various spaces of Dirichlet series
(for which the set $X$ would be a half-plane),
but with the proviso that the sequences are bounded.
In \cite{mcsh15}, spaces of Dirichlet series with the complete Pick property, including universal ones, are constructed. Theorem~\ref{CNPintThm}
describes all their interpolating sequences.
\end{remark}

\begin{remark}
The first part of the proof of Theorem \ref{CNPintThm} shows the
 existence of $\Psi$
satisfying \eqref{eqd1}, given $\Phi$ satisfying \eqref{eqd2}.
If we could show that $\Phi^t$, the transpose of $\Phi$, were bounded, we would not need to
invoke the Marcus-Spielman-Srivastava theorem.
T. Trent proved this for the Dirichlet space on the disk \cite{tre04}.
The following theorem, whose proof the authors intend to give in an upcoming paper,
shows that it is also true on the Drury-Arveson space of finitely many variables.
\end{remark}

\begin{theorem} \label{thmd3} Let  $\h$ be the Drury-Arveson space on the ball $\B_d$, for $d$ finite.
 Then there is $C>0$ such that whenever $\{\varphi_i\}$ is a sequence of functions  in ${\rm Hol}(\B_d)$
with $$\sum_{i=1}^\infty \|\varphi_i h\|^2 \le \|h\|^2 \text{ for all }h \in \h_,$$ then
$$\|\sum_{j=1}^\infty \varphi_j h_j\|^2 \le C \sum_{j=1}^\infty \|h_j\|^2 \text{ for all } \{h_j\}\in \h\otimes \ell^2.$$
\end{theorem}

\begin{remark}
It is easy to show \cite[Proposition 9.11]{ampi} that if the Gram matrix is bounded, then the sequence is a finite union
of sequences that have (BG) and (S). So whenever the atomic measure $\sum \| \cp_{\l_i} \|^2 \delta_{\l_i}$ is a Carleson measure (\ie (C) holds), the sequence $(\l_i)$ is a finite union of interpolating sequences.
Carleson measures for the Drury-Arveson space were characterized in \cite{ars08}, and in a different way in
\cite{vw12}.
\end{remark}

\section{Pairs of kernels}
\label{sece}

If  $k,\ell$ are two reproducing kernels on a set $X$, then one can formulate a Pick problem for $\m(\h_k,\h_\ell)$: For which $\lambda_1,...,\lambda_n\in X$ and $w_1,..., w_n\in \C$ is there a multiplier $\varphi \in \m(\h_k,\h_\ell)$ with $\|\varphi\|_M\le 1$ and such that $\varphi(\lambda_i)=w_i$ for $i=1,...,n$?

 An obvious necessary condition is obtained by looking at $M_\varphi^*: \h_\ell \to \h_k$.
We have  $M^*_\varphi \ell_\lambda = \overline{\varphi(\lambda)} k_\lambda$ for all $\lambda\in X$, and one checks that $\|M^*_\varphi\|\le 1$ if and only if $\ell_\lambda(z) -\varphi(z)\overline{\varphi(\lambda)}k_\lambda(z)$ is positive semi-definite. Thus, if $\|M_\varphi\|\le 1$ and $\varphi(\lambda_i)=w_i$ for $i=1,...,n$, then by restricting $M^*_\varphi$ to the linear span of $\ell_{\lambda_1},..., \ell_{\lambda_n}$ one sees that the matrix $$\left(\ell_{\lambda_i}(\lambda_j) -w_j\overline{w}_i k_{\lambda_i}(\lambda_j)\right)_{1\le i,j\le n}$$ is positive semi-definite. We say that the pair $(k,\ell)$ has the Pick property if this condition is always sufficient for the solution of the Pick problem in this context.

A straightforward compactness argument shows that if the pair $(k,\ell)$ has the Pick property,
then one can solve Pick problems with infinitely many points.
More precisely if $(k,\ell)$ has the Pick property, then
whenever an operator $T$ is defined on a subspace of $\h_\ell$ spanned by the kernel functions
from some set $\F \subseteq X$
in the following way
\beq
T: \vee  \{ \ell_\l  : \ \l \in \F \} & \ \to \ & \ \h_k \\
 \hl_\l  &\ \mapsto & \overline{\alpha(\l)} {\hkl},
\eeq
and $T$ is bounded, then $\alpha$ extends to a (necessarily bounded) function on all of $X$ so that
the extension of $T$ given by
$$
\hl_\l \ \mapsto \overline{\alpha(\l)} {\hkl} \quad \forall\  \l \in X
$$
 has the same norm as $T$. Note that for the extension we have $T=M_\varphi^*$ for $\varphi \in \m(\h_k,\h_\ell)$ with $\varphi(\lambda)= \alpha(\lambda)\frac{\|\ell_\lambda\|}{\|k_\lambda\|}$.

\begin{proposition}
\label{prope1}
(i) If $(\l_i)$ is an interpolating sequence for $\m(\h_k,\h_\ell)$, then the Gram matrix $G(k)$ is bounded, and $G(\ell)$
is bounded below.

(ii) If $(k,  \ell)$ has the
 Pick property, the converse holds.
\end{proposition}
\bp
Let $\nu(\lambda) = ||\ell_{\lambda}|| / ||k_{\lambda}||$, and observe
that a bounded operator $T: \mathcal H_\ell \to \mathcal H_k$ maps
$\hat \ell_{\lambda}$ to $\overline{\alpha(\lambda)} \hat k_{\lambda}$ if and only
if the function $\varphi(\lambda) = \nu(\lambda) \alpha(\lambda)$ belongs
to $\m(\mathcal H_k,\mathcal H_\ell)$ and $T = M_\varphi^*$.

  (i) Since $(\lambda_i)$ is interpolating, the bounded linear map
  \begin{equation*}
    \m(\mathcal H_k, \mathcal H_\ell) \to \ell^\infty, \quad \varphi \mapsto (\nu(\lambda_i)^{-1} \varphi(\lambda_i)),
  \end{equation*}
  is surjective. An application of the open mapping theorem therefore shows that there exists
  $M > 0$ such that for every  sequence $(\alpha_i)$ in the unit ball of $\ell^\i$ and every
  finitely supported sequence $(c_i)$ of complex numbers, we have
\[
\Big\| \sum_i c_i \bar \alpha_i \hat{k}_{\l_i} \Big\|_{\h_k}^2 \ \leq \ M
\Big\| \sum_i c_i \hat{\ell}_{\l_i} \Big\|_{\h_\ell}^2 .
\]
Choosing $\alpha_j = \exp(it_j)$ and integrating w.r.t. each $t_j$, we get $G(\ell)$ is bounded below;
letting $c_j = \exp(it_j) a_j$ and $\alpha_j = \exp(it_j)$ and integrating, we get
$G(k)$ is bounded above (cf. the proof of \cite[Theorem 9.19]{ampi}).

(ii) Suppose there are constants $M_1, M_2 > 0$ so that
\[
M_1 \Big\| \sum_j a_j \hat{k}_{\l_j} \Big\|^2 \ \leq \ \sum_j |a_j|^2 \ \leq \ M_2
 \Big\| \sum_j a_j \hat{\ell}_{\l_j} \Big\|^2
\]
holds for every finitely supported sequence of complex numbers $(a_j)$.
Let $(\alpha_j)$ belong to the unit ball of $\ell^\i$. Then
for every finitely supported sequence of complex numbers $(a_j)$, we have
\begin{equation*}
  \Big\| \sum_{j} a_j \bar{\alpha_j} \hat k_{\lambda_j} \Big\|^2
  \le \frac{1}{M_1} \sum_j |a_j \bar{\alpha_j}|^2 \le \frac{1}{M_1} \sum_j |a_j|^2
  \le \frac{M_2}{M_1} \Big\| \sum_j a_j \hat \ell_{\lambda_j} \Big\|^2,
\end{equation*}
hence there exists a unique bounded linear operator $R$ from $\vee \{ \ell_{\lambda_i}: i \in \mathbb N \}$
into $\mathcal H_{k}$ which maps $\hat \ell_{\lambda_j}$ to $\bar \alpha_j \hat k_{\lambda_j}$.
Since the pair has the Pick property, this implies that there is a $\varphi\in \m(\h_k,\h_\ell)$ with $\varphi(\l_i)=\alpha_i \frac{\|\ell_{\l_i}\|}{\|k_{\l_i}\|}$ for all $i$. Thus,
$(\l_i)$ is interpolating for $\m(\h_k,\h_\ell)$.
\ep

For the remainder of this
 section, $\cp$ will be an irreducible complete Pick kernel normalized at some point, and
\be
\label{eqdx1}
\ell(z,w) \= \cp(z,w) g(z,w) \ee
where $g$ is also a kernel.
As mentioned in the Introduction for a concrete example, think of $\cp$ as the \sz kernel, and $\ell$ as the Bergman kernel on the disk.

As $g$ is a kernel, there is a map $\Gamma : X \to \M^*$, for some Hilbert space $\M$, so that
\be
\label{eqxg}
g(z,w) \= \Gamma(z) \Gamma(w)^* .
\ee

\begin{proposition}
\label{prope2}  Let $\cp$ and $\ell$ be as in \eqref{eqdx1}.
Then the pair $(\cp, \ell)$ has the Pick property.
\end{proposition}

\bp
Suppose $\F \subseteq X$, and $T$ is a  contraction defined by
\beq
T: \vee \{ \ell_\lambda \ : \ \l \in \F \} & \to & \h_s \\
\hat{\ell}_\l &\mapsto & \overline{\alpha(\l)} \hat{\cp}_\l.
\eeq
Let $\nu(\l) = \| \ell_\lambda \| / \| \cp_\l \|$.
Then
\be
\label{eqe4}
\left[ \Gamma(z) \Gamma(w)^* -  \alpha(z) \nu(z) \overline{\alpha(w) \nu(w)} \right] \cp(z,w) \geq  0.
\ee
So by Leech's theorem (also known as the Toeplitz corona theorem) for complete Pick kernels
\cite{at03, btv01}, see also \cite[Theorem 8.57]{ampi}, there is a function $\Delta$ of norm at most one, defined on all
of $X$, such that
\be
\label{eqe5}
 \Gamma(z)\Delta(z) \= \nu(z) \alpha(z)  \quad \forall z \in \F.
\ee
Use \eqref{eqe5} to extend $\alpha$ to all of $X$; and then
\eqref{eqe4} will hold on $X \times X$.
This extended $\alpha$ will then give the desired extension of $T$.
\ep

Let us find $\Delta$ explicitly.
From \eqref{eqe4}, we can  represent the left hand-side as
$\Theta(z) \Theta(w)^*$, where $\Theta$ takes values in $B(\L,\C)$
for some auxiliary Hilbert space $\L$:
\[
\left[ \Gamma(z) \Gamma(w)^* -  \alpha(z) \nu(z) \overline{\alpha(w) \nu(w)} \right] \cp(z,w)
\= \Theta(z) \Theta(w)^* .
\]
  Then,
 writing $\phi(z) = \alpha(z) \nu(z)$ and using Theorem~\ref{thmb1},
 we get for $z,w$ in $\F$:
\be
\label{eqe9}
\left[  \Gamma(z) \Gamma(w)^* - {\phi(z)}\overline{\phi(w)} \right]
\=
\Theta(z) \Theta(w)^*
(1 - \la b(z) , b(w) \ra) .
\ee
Define
\begin{eqnarray*}
\beta(z) : \L & \to & \ell^2_d \otimes \L \\
u & \mapsto & \overline{b(z)} \otimes u .
\end{eqnarray*}
Let
$ \E(z) := \beta(z)^*$, so
$$ \E(z) : \sum \xi_j \otimes u_j \mapsto \sum \langle \xi_j , \overline{b(z)} \ra u_j .
$$
A lurking isometry argument on \eqref{eqe9} gives an isometry
\[
U
:
\begin{pmatrix}
\Gamma(w)^* \\
\beta(w)  \Theta(w)^*
\end{pmatrix}
\ \mapsto \
\begin{pmatrix}
{\overline{\phi(w)}} \\  \Theta(w)^*
\end{pmatrix},
\]
where the domain of $U$ is the span of the vectors on the left as $w$ ranges over $\F$.
By enlarging the auxiliary Hilbert space $\mathcal L$ if necessary, we may extend $U$ to a unitary on
the whole space, and write it as
\[
U
\=
\begin{pmatrix}
A^*&C^*\\B^*&D^*
\end{pmatrix} .
\]

We get that for $w \in \F$,
\be
\label{eqe10}
\overline{{\phi(w)}} \= [ A^* + C^* \beta(w) ( I - D^* \beta(w) )^{-1} B^* ] \Gamma(w)^*.
\ee

Define $\Delta(w)$, for all $w$ in $X$, to be the adjoint of the quantity in brackets in \eqref{eqe10}:
\be
\label{eqe13}
\Delta(w) \= A+ B \E(w) ( I -   D \E(w))^{-1}   C ,
\ee
so
\[
\phi(w) = \Gamma(w) \Delta(w).
\]

Since $U$ is a unitary, we get
\[
[I - \Delta(z) \Delta(w) ^* ]\ s(z,w) \geq \
B ( I - \E(z) D)^{-1}
 (I - D^* \E(w)^*)^{-1} B^* ,
\]
and hence $\| \Delta \| \leq 1 $ as a multiplier from
$\h_s$ to $\M \otimes \h_s$.

So we have proved a realization formula.
Taking $\F = X$ we get:
\begin{proposition}
Every contractive multiplier $\phi$ in ${\rm Mult}(\h_\cp, \h_\ell)$ can be represented as
 in $\phi(w) = \Gamma(w) \Delta(w)$, where $\Gamma$ is as in \eqref{eqxg}, and
 $\Delta$ is a contractive multiplier from $\hs$ to $\M \otimes \hs$, given by a realization formula
 \eqref{eqe13}.
\end{proposition}

For $\cH_s = H^2$ and certain weighted Hardy spaces $\cH_{\ell}$,
this result was previously obtained (with essentially the same proof) by Ball
and Bolotnikov \cite[Theorem 2.1]{BB17}.

In the proof of Theorem \ref{MixedInterpolation}, we require the following lemma.
\begin{lemma} \label{GrammianLemma2} Let $\cp$ and $\ell$ be as above, and let $(\lambda_i)\subseteq X$ be any sequence.

(a) If $G(\cp)$ is bounded, then $G(\ell)$ is bounded. Consequently, if $(\lambda_i)$ satisfies the Carleson condition for $\h_{\cp}$, then it satisfies the Carleson condition for $\h_\ell$.

(b) Similarly, if $G(\cp)$ is bounded below, then $G(\ell)$ is bounded below.

(c) If $(\lambda_i)$ is interpolating for $\h_\cp$, then it is interpolating for $\h_\ell$.

\end{lemma}
\begin{proof} Note that the Gram matrix $G(g)$ is positive and that the Schur product of $G(g)$ with the identity matrix $I$ equals $I$. Hence (a) and (b) of the Lemma follow from the Schur product theorem and the observation that $G(\ell)$ equals the Schur product of $G(g)$ and $G(s)$. For example, the boundedness of $G(\cp)$ implies that there is $c>0$ such that $cI-G(\cp)$ is positive semi-definite, hence taking the Schur product with $G(g)$ we obtain $cI- G(\ell)$ is positive semi-definite, hence $G(\ell)$ is bounded. The statements about the Carleson conditions and part (c) now follow from Lemma \ref{elementaryGrammian}.
\end{proof}

We will now prove Theorem \ref{MixedInterpolation}.

\bp (a) By Proposition \ref{prope2} the pair $(\cp, \ell)$ has the Pick property, hence Proposition \ref{prope1} implies that a sequence is interpolating for $\m(\h_{\cp},\h_\ell)$ if and only if $G(\cp)$ is bounded above and $G(\ell)$ is bounded below. Now part (a) follows from Lemmas \ref{elementaryGrammian} and \ref{GrammianLemma2}.

(b) If a sequence is interpolating for $\m(\h_s)$, then it must be interpolating for $\h_s$. Thus  Lemmas \ref{elementaryGrammian} and \ref{GrammianLemma2} imply that it is interpolating for $\h_\ell$ and that it satisfies the Carleson condition for $\h_s$. Thus by (a) the sequence is interpolating for $\m(\h_s,\h_\ell)$. Conversely, if an $\h_s$ separated sequence is  interpolating for $\m(\h_s,\h_\ell)$, then by (a) it satisfies the Carleson condition (C) for $\h_s$ and hence it follows from Theorem \ref{CNPintThm} that it must be interpolating for $\m(\h_s)$.
\ep

\vskip 10 pt
We now turn to the proof of Theorem \ref{thmsst}. For the reader's convenience we restate the theorem here.

\begin{theorem} \label{thmsstAgain} Let $s_1,s_2$ be  normalized complete Pick kernels on $X$ such that $s_2/s_1$ is positive semi-definite, and  let $t\ge 1$.

Then a sequence is interpolating for $\m(\h_{s_1},\h_{s_2^t})$ if and only if it satisfies the $\h_{s_1}$ Carleson condition (C) and is $\h_{s_2}$ separated (S).
\end{theorem}
\bp Note that for $t \ge 1$ \cite[Remark 8.10]{ampi} implies that $s_2^{t-1}$ is positive semi-definite, hence the hypothesis and the Schur product theorem imply that $g= s_2^t/s_1$ is positive semi-definite and we can apply Theorem \ref{MixedInterpolation} (a).

Suppose a sequence is $\m(\h_{s_1},\h_{s_2^t})$ interpolating, then by Theorem \ref{MixedInterpolation} it satisfies the Carleson condition for $\h_{s_1}$ and it is interpolating for $\h_{s_2^t}$. But then it must be  $\h_{s_2^t}$ separated, which is the same as being $\h_{s_2}$ separated.

Conversely, assume a sequence satisfies the Carleson condition for $\h_{s_1}$ and it is  $\h_{s_2}$ separated. By Theorem \ref{MixedInterpolation} we must show that the sequence is interpolating for $\h_{s_2^t}$.
By Lemma  \ref{GrammianLemma2} the sequence satisfies the $\h_{s_2}$ Carleson condition, thus by Theorem \ref{CNPintThm} it must be interpolating for $\m(\h_{s_2})$ and for $\h_{s_2}$. But now since $s_2^t=s_2^{t-1}s_2$ we can apply Lemma \ref{GrammianLemma2} (c) to conclude that the sequence is interpolating for $\h_{s_2^t}$.
\ep

The following is an example, where Theorem \ref{MixedInterpolation} applies, but Theorem \ref{thmsst} does not. In the case of the example there are $\m(\h_s,\h_{\ell})$ interpolating sequences that are not $\m(\h_s)$ interpolating.

\begin{example} \label{ExNotSeparated}
{\rm
Let $\h_s=H^2$, let $\mathcal{H}_\ell$ be the weighted Bergman space with weight $\exp(-\frac1{1-|z|^2})$, and let
$$z_j=1-2^{-j}, \quad w_j=z_j+i2^{-5j/4},\,\,j\ge 0.$$
By \cite[Theorem 2.4]{bdk07}
 the union of $(z_j)$ and $(w_j)$ is interpolating for $\mathcal{H}_\ell$, so $G(\ell)$ is bounded below.
This union is not weakly separated in $H^\i(\D)$, but both sequences $(z_j)$ and $(w_j)$ are interpolating for $H^2$,
so $G(\cp)$ is bounded.
}
\end{example}

\section{Applications and Questions}
\label{secf}

S. Axler  proved in \cite{ax92} that the multiplier algebra of the Dirichlet space has infinite interpolating sequences.
He did this by using the Rosenthal-Dor theorem to prove that any sequence either has a subsequence that is interpolating, or a subsequence along which every function in the multiplier algebra has a limit; then he exhibited sequences in which the latter does not happen.
Not all Hilbert function spaces with complete Pick kernels admit infinite interpolating sequences. If $\alpha \ge 0$, then Kaluza's lemma (see \cite{SS61} or \cite[Corollary 7.41]{ampi}) implies that $s_w(z) =\sum_{n=0}^\infty \frac{(\overline{w}z)^n}{(n+1)^\alpha}$ defines a normalized irreducible complete Nevanlinna Pick kernel on the unit disk. If $\alpha >1$, then $\sup_{w\in \D}\|s_w\|<\infty$, and hence the Carleson condition (C) can never be satisfied for the constant function $f(z)=1 \in \h_s$.

 It is also easy to see that the converse to this last observation is true;
 special cases of this result also appear in \cite[Lemma 8.2]{Hartz15} and
 \cite[Proposition 9.1]{dhs15}.

\begin{proposition}
  \label{prop:interpolating_subsequence}
  Let $s$ be a normalized complete Pick kernel on a set $X$, and  $(\lambda_i)\subseteq X$.
  If  $\|s_{\lambda_i}\| \to \infty$, then $(\lambda_i)$ has a subsequence which is interpolating for $\h_s$ and for $\m(\h_s)$.
\end{proposition}

 \begin{proof}
   Assume  $\|s_{\lambda_i}\| \to \infty$.
  By Theorem \ref{thmb1}, we may write $s_w(z) = (1 - \langle b(z),b(w) \rangle)^{-1}$
  with $||b(w)|| \le 1$ for all $w \in X$, so
  for all $z,w\in X$ we have $|s_w(z)|\le (1-\|b(w)\|)^{-1}$. Thus, since finite linear combinations of reproducing kernels are dense in $\h_s$, it follows that the normalized kernels $\hat{s}_{\lambda_i}$ converge to 0 weakly. It is now easy to see that $\hat{s}_{\lambda_i}$ has a subsequence that is a Riesz sequence. Thus the conclusion follows from Lemma \ref{elementaryGrammian}.
 \end{proof}

In \cite{LechEssNormal} Lech shows the existence of a multiplication operator $M_\phi$ on the Dirichlet space that is not essentially normal.
His result can be generalized. We first require a routine lemma.
Recall 
that if $\cH$ is a reproducing kernel Hilbert space on $X$ with kernel $k$, the
pseudo metric $d_{\cH}$ on $X$ is defined by
\begin{equation*}
  1 - d_{\cH}(z,w)^2 = \frac{|\langle k_z,k_w \rangle|^2}{||k_z||^2 \, ||k_w||^2}, \quad
  z,w \in X.
\end{equation*}

\begin{lemma}
  \label{lem:sequences_to_infinity}
  Let $s$ be a normalized complete Pick kernel on a set $X$
  and let $(z_n)$ be a sequence in $X$ with $||s_{z_n}|| \to \infty$.
  \begin{enumerate}
    \item[(a)] For all $z \in X$, we have $\lim_{n \to \infty} d_{\cH}(z_n,z) = 1$.
    \item[(b)] If $(w_n)$ is another sequence in $X$ such that there exists $r < 1$ with
      $d_{\cH} (z_n,w_n) < r$ for all $n \in \bN$,
      then $\lim_{n \to \infty} ||s_{w_n}|| = \infty$.
  \end{enumerate}
\end{lemma}

\begin{proof}
  The proof of Proposition \ref{prop:interpolating_subsequence} shows that the normalized kernels
  $\widehat s_{z_n}$ tend to $0$ weakly, so that
  $1 - d_{\cH} (z_n,z)^2 = |\langle \widehat s_{z_n}, \widehat s_z \rangle|$ tends to $0$, which establishes (a).
  
  To prove (b), note that
  \begin{equation*}
    |s(z_n,w_n)|^2 = s(z_n,z_n) s(w_n,w_n) ( 1 - d_{\cH}(z_n,w_n)^2)
    \ge s(z_n,z_n) (1 - r^2),
  \end{equation*}
  which tends to $\infty$. Therefore,
  writing
  $s_w(z) = (1 - \langle b(z),b(w) \rangle)^{-1}$
  with $||b(w)|| \le 1$ for all $w \in X$ (see Theorem \ref{thmb1}), we see that
  $\langle b(z_n),b(w_n) \rangle$ tends to $1$, which forces $||b(w_n)||$ to tend to $1$,
  so that $s(w_n,w_n)$ tends to $\infty$.
\end{proof}

Every interpolating sequence is separated in the pseudo metric $d_{\cH}$.
The following lemma allows us to find an interpolating sequence such that infinitely many
pairs of points have small distance.
\begin{lemma}
  \label{lem:two_sequences}
  Let $X$ be a connected topological space and let $\cH$ be a normalized complete Pick space on $X$ with jointly
  continuous kernel $s$. If $s$ is unbounded, then
  for every $\varepsilon > 0$, there exist sequences $(z_n)$ and $(w_n)$ in $X$
  whose union is interpolating such that
  $d_{\cH}(z_n,w_n) < \varepsilon$ for all $n \in \bN$.
\end{lemma}

\begin{proof}
  For ease of notation, we will write $d = d_\cH$.
  Since $s$ is unbounded, there exists an interpolating sequence $(z_n)$ for $\Mult(\cH)$
  by Proposition \ref{prop:interpolating_subsequence}.
  In particular, there exists $\delta > 0$ such that $d(z_n,z_m) > 5 \delta$ if $n \neq m$.
  We may without loss of generality assume that $2 \delta < \varepsilon$
  and that $2 \delta < 1$.

  If $z \in X$, then continuity of $s$ shows that the map
  \begin{equation*}
    f: X \to [0,1), \quad x \mapsto d(x,z),
  \end{equation*}
  is continuous. Clearly, $f(z) = 0$ and $\lim_{n \to \infty} f(z_n) = 1$ by part (a)
  of Lemma \ref{lem:sequences_to_infinity},  so connectivity of $X$ implies that $f$
  is surjective.
  Therefore, for $n \in \bN$, we may choose
  \begin{equation*}
    w_n \in \{ x \in X: \delta < d(x,z_n) < 2 \delta\}.
  \end{equation*}
  Then $d(z_n,w_n) < \varepsilon$ for all $n \in \bN$.
  Since $s(z_n,z_n)$ tends to $\infty$ and since $s(z_n,w_n) < 2 \delta < 1$, part (b)
  of Lemma \ref{lem:sequences_to_infinity} shows that $s(w_n,w_n)$ tends to infinity
  as well.
  By passing to a subsequence simultaneously for $(z_n)$
  and $(w_n)$, we may assume that $(w_n)$ is also interpolating by Proposition \ref{prop:interpolating_subsequence}.

  It is straightforward to check that
  the union of the sequences $(z_n)$ and $(w_n)$ is $d_\cH$-separated by $\delta$.
  Moreover, since each sequence is individually interpolating, they both satisfy the Carleson condition,
  hence so does their union. An application of Theorem \ref{CNPintThm} now shows that the union of $(z_n)$
  and $(w_n)$ is interpolating.
\end{proof}

The following result is the desired generalization of Lech's theorem.
\begin{proposition}
\label{prxe1}
  Let $X$ be a connected topological space and let $\cH$ be a normalized complete Pick space on $X$ with jointly
  continuous kernel $s$. If $s$ is unbounded, then
  there there exists a multiplication operator on $\cH$ which is not essentially normal.
\end{proposition}

\begin{proof}
  We apply Lemma \ref{lem:two_sequences} with $\varepsilon = 1/2$ to find
  sequences $(z_n)$ and $(w_n)$ whose union is interpolating with
  \begin{equation*}
    d_{\cH}(z_n,w_n) < 1/2
  \end{equation*}
  for all $n \in \bN$. In particular, there exists a multiplier
  $\varphi \in \Mult(\cH)$ with $\varphi(z_n) = 0$
  and $\varphi(w_n) = 1$ for all $n \in \bN$. We claim that $M_\varphi$ is not essentially normal.

  Since $(z_n)$ is interpolating, $s(z_n,z_n)$ tends to $\infty$,
  which implies that the sequence of normalized kernels $\widehat s_{z_n}$ converges
  to zero weakly in $\cH$ by the proof of Proposition \ref{prop:interpolating_subsequence}.
  Therefore, it suffices to prove that
  \begin{equation*}
    ||(M_\varphi^* M_\varphi - M_\varphi M_\varphi^*) \widehat s_{z_n}||
  \end{equation*}
  does not converge to zero. To this end, we use the fact that $\varphi(z_n) = 0$ to see that the
  sequence above is bounded below by
  \begin{equation*}
    |\langle (M_\varphi^* M_\varphi - M_\varphi M_\varphi^*) \widehat s_{z_n}, \widehat s_{z_n} \rangle| =
    ||M_\varphi \widehat s_{z_n}||^2.
  \end{equation*}
  Using further
  that $\varphi(w_n) = 1$, we see that this quantity, in turn, is bounded below
  by
  \begin{equation*}
    |\langle M_\varphi \widehat s_{z_n}, \widehat s_{w_n} \rangle|^2
    = |\langle \widehat s_{z_n}, \widehat s_{w_n} \rangle|^2
    = 1 - d_{\cH} (z_n,w_n)^2 \ge 1 - \frac{1}{4} = \frac{3}{4},
  \end{equation*}
  which finishes the proof.
\end{proof}

Suppose that $s, \ell$ are reproducing kernels in the open unit disc $\D$ such that $\ell/s$ is positive semi-definite and $s$ is a normalized complete Pick kernel such that $z \in \Mult(\mathcal H_s)$. Assume further that there is $0\ne f\in \h_s$ such that for every $g \in \h_\ell$ we have $\|z^n g\|_{\h_\ell}/\|z^n f\|_{\h_s} \to 0$ as $n\to \infty$. This is satisfied if e.g. $\h_\ell$ is a weighted Bergman space and $s$ is the Szeg\"{o} kernel or the kernel for the Dirichlet space.

Then no $\varphi \in \m(\h_s,\h_\ell)$ can define a multiplication operator $\h_s\to \h_\ell$ with closed range. That is because such a  multiplication operator would have to be bounded below and  this would contradict the hypothesis
$$0< c \le \frac{\|z^n \varphi f\|_{\h_\ell}}{\|z^n f\|_{\h_s}}\to 0 \ \ \text{ as }n \to \infty.$$

Nevertheless it follows from the results on interpolating sequences between spaces that as long as $\|s_z\|$ is unbounded on $\D$ there always are non-compact multiplication operators $\h_s\to \h_\ell$. In the case where $\h_s=D$ is the Dirichlet space and $\h_\ell=L^2_a$ is the Bergman space this is Lech's result, \cite[Theorem 2]{LechEssNormal}. By use of the results of Section \ref{sece} we can give a short proof.

Indeed, if $(\lambda_i)$ is an interpolating sequence for $\h_s$, then by Lemma \ref{GrammianLemma2} (c),
 it must also be an interpolating sequence for $\h_\ell$. That means that $\hat{s}_{\lambda_i}$ is a Riesz sequence for $\h_s$ and $\hat{\ell}_{\lambda_i}$ is a Riesz sequence for $\h_\ell$. Thus the operator $T: \bigvee_i \{\hat{\ell}_{\lambda_i}\} \to \bigvee_i \{\hat{s}_{\lambda_i}\}$ defined by $\hat{\ell}_{\lambda_i}\to \hat{s}_{\lambda_i}$ is bounded and invertible, hence it is noncompact. Since $(s,\ell)$ has the Pick property $T$ can be extended to $M_\varphi^*$ for some $\varphi\in \m(\h_s,\h_\ell)$. Clearly $M_\varphi$ is not compact.

A positive answer to the following question would say that the analogue of Theorem \ref{CNPintThm} holds in the context of Theorem \ref{MixedInterpolation}. Theorem \ref{thmsst} and Example \ref{ExNotSeparated} provide situations where this is the case.
\begin{question} If $s$ is a normalized complete Pick kernel, and if $\ell=g s$ for some positive semi-definite kernel $g$, then is it true that a sequence $(\lambda_i)\subseteq X $ is interpolating for $\m(\h_s,\h_\ell)$ if and only if  $(\lambda_i) $ is a Carleson sequence for $\h_s$ and is  $\h_\ell$  separated?
\end{question}

\bibliography{references}

{\it Email of corresponding author: mccarthy@wustl.edu}
\end{document}